\documentclass{article}
\usepackage{amsmath,amsthm,amssymb, mathtools}
\usepackage{graphicx,pgfplots,tikz}
\usepackage[colorinlistoftodos]{todonotes}
\usepackage[colorlinks=true, allcolors=blue]{hyperref}

\usepackage[english]{babel}
\usepackage[utf8x]{inputenc}
\usepackage[T1]{fontenc}
\usepackage{amsmath}
\usepackage{verbatim}

\newtheorem{definition}{Definition}[section]
\newtheorem{theorem}[definition]{Theorem}
\newtheorem{lemma}[definition]{Lemma}
\newtheorem{mlemma}[definition]{Main Lemma}
\newtheorem{example}[definition]{Example}
\newtheorem{gevolg}[definition]{Corollary}              
\newtheorem{commentaar}[definition]{Remark}

\renewenvironment{proof}[1][\noindent Proof]{{\par\pushQED{\qed}%
\itshape #1\@. }}{\popQED}

\newcommand{\qbin}[2]{\genfrac{[}{]}{0pt}{}{#1}{#2}}

\newcommand{\ds}{\displaystyle}
\newcommand{\cS}{\mathcal{S}}
\newcommand{\cP}{\mathcal{P}}
\newcommand{\cL}{\mathcal{L}}
\definecolor{blauw}{rgb}{0,0,0.6}

\newcommand{\PG}{\mathrm{PG}}   
\title{On the sunflower bound for $k$-spaces, %
pairwise intersecting in a point}
\author{A. Blokhuis\thanks{Department of Mathematics and Computer Science, Eindhoven University of Technology, The Netherlands, \emph{a.blokhuis@tue.nl}. }, M. De Boeck\thanks{Department of Mathematics: Algebra and Geometry, Ghent University, Flanders, Belgium, \emph{maarten.deboeck@ugent.be}.},  J. D'haeseleer\thanks{(corresponding author) Department of Mathematics: Analysis, Logic and Discrete Mathematics, Ghent University, Flanders, Belgium, \emph{jozefien.dhaeseleer@ugent.be}.}}
\date{}
\begin{document}
\maketitle
\begin{abstract}
A $t$-intersecting constant dimension subspace code $C$ is a set of $k$-dimensional subspaces in a projective space $\PG(n,q)$, where distinct subspaces intersect in exactly a $t$-dimensional subspace. A classical example of such a code is the sunflower, where all subspaces pass through the same $t$-space. The sunflower bound states that such a code  is a sunflower if
$|C| > \left( \frac {q^{k + 1} - q^{t + 1}}{q - 1} \right)^2 +
\left( \frac {q^{k + 1} - q^{t + 1}}{q - 1} \right) + 1$.

In this article we will look at the case $t=0$ and we will improve this bound for $q\geq 9$: a set $\cS$ of  $k$-spaces in $\PG(n,q), q\geq 9$, pairwise intersecting in a point is a sunflower if  $|\mathcal{S}|> \left(\frac{2}{\sqrt[6]{q}}+\frac{4}{\sqrt[3]{q}}-\frac{5}{\sqrt{q}}\right)\left(\frac {q^{k + 1} - 1}{q - 1}\right)^2$.
\end{abstract}


\section{Introduction}
In a vector space $V$, a $(k,t)$-SCID is a set of $k$-dimensional subspaces of $V$, pairwise intersecting in exactly a $t$-dimensional subspace (SCID stands for:
\textit{Subspaces with Constant Intersection Dimension}). We will work in a projective context, so a $(k+1,t+1)$-SCID corresponds to a set of $k$-dimensional projective subspaces in $\PG(n,q)$, that pairwise intersect in a $t$-dimensional space (see \cite{Eisfeld}). A $(k+1,t+1)$-SCID is also called a $t$-intersecting constant dimension subspace code, where the code words have projective dimension $k$. Note that $(k+1,0)$-SCIDs correspond with partial $k$-spreads in $\PG(n,q)$.

An example of a $(k+1,t+1)$-SCID is a sunflower, which is a set of
$k$-spaces, passing through the same $t$-space and having no points in common outside this $t$-space. It can be shown that a $t$-intersecting constant dimension subspace code is a sunflower if the code has many code words.

\begin{theorem}\cite[Theorem 1]{etzion}\label{sunny}
A $(k+1,t+1)$-SCID $C$ is a sunflower if
\begin{align*}
|C| > \left( \frac {q^{k+1} - q^{t+1}}{q - 1} \right)^2 +
\left( \frac {q^{k+1} - q^{t+1}}{q - 1} \right) + 1.
\end{align*}
\end{theorem}
It is believed that the sunflower bound is in general not tight.
In \cite{sunflowern-2}, the sunflower bound for $(k+1,1)$-SCIDs was studied.
\begin{theorem}[{\cite[Theorem 2.1]{sunflowern-2}}]
	Let $C$ be a $(k+1,1)$-SCID, with $k\geq 4$. If 
	\begin{align*}
	|C|\geq \left( \frac {q^{k+1} - q}{q - 1} \right)^2 +
	\left( \frac {q^{k+1} - q}{q - 1} \right) -q^{k},
	\end{align*} then $C$ is a sunflower.
\end{theorem}
In this article we will give a better result for $(k+1,1)$-SCIDs.
In \cite{leo} a geometrical sunflower bound was studied: the authors investigate SCIDs that span a large subspace. They prove that sunflowers are the SCIDs that span the `largest' subspaces. 

We suppose that $k \geq 3$ as for $( 2, 1 )$-SCIDs 
we, more generally, known that every $( k+1, k)$-SCID is a sunflower or consists of all $k$-spaces in a fixed $(k+1)$-space.
For $( 3, 1)$-SCIDs, an almost complete classification is known, see \cite{beutel}.

In Section \ref{sectieprelimi}, we give some definitions and general lemmas.
In Section \ref{section3}, we start with the Main Lemma that gives an
important inequality. Using this inequality we continue with
Theorem \ref{maintheorem} that gives
an improvement on the sunflower bound if $k\geq 3$ and $ q\geq 9$ (and if $q\geq 7$ and $k\geq 5$).

\section{Preliminaries} \label{sectieprelimi}
We start with the definition of the $q$\textit{-ary Gaussian coefficient}.

\begin{definition}
	Let $q$ be a prime power,  let $n,k$ be  non negative integers with $k \leq n $. The $q$\textit{-ary Gaussian
		coefficient} of $n$ and $k$ is defined by
	\begin{equation*}
	\qbin{n}{k}_q=
	\begin{cases} 
	\frac{(q^n-1)\cdots(q^{n-k+1}-1)}{(q^k-1)\cdots(q-1)}  \hspace{1.3cm}\textnormal{ if $k>0$} \\
	\hspace{1.5cm} 1 \hspace{2.5cm}\textnormal{otherwise} 
	\end{cases}
	\end{equation*}
\end{definition}
We will write $\qbin{n}{k}$, if the field size $q$ is clear from the context. The number of $k$-spaces in $\PG(n,q)$ is $\qbin{n+1}{k+1}$  and the number of $k$-spaces through a fixed $t$-space  in $\PG(n,q)$, with $0 \leq t \leq k$, is $\qbin{n-t}{k-t}$.
Moreover, we will denote the number $\qbin{n+1}{1}$ by the symbol $\theta_n$.

From now on we consider a fixed $(k+1,1)$-SCID $\cS$ that is {\em not} a sunflower, of size $|\cS|=(1 - s) \theta_{k}^2$, $0 < s < 1$. Note that the size of $|\cS|$ is smaller than the sunflower bound for $s > \frac {1}{\theta_k}-\frac {1}{\theta_k^2}$. We will give, for a fixed value of $k$ and field size $q$, an upper bound on $1 - s$. For convenience we will call a $k$-space contained in $\cS$ a {\em block}.

\begin{definition}
Consider the above set $\cS$ of $k$-spaces, or blocks.
The sets of points and lines that are contained in a block are denoted by $\cP_\cS$ and $\cL_\cS$ respectively.
\end{definition}

\begin{lemma}\label{puntoprecht}
Suppose $P \in \cP_\cS$, then $P$ lies in at most $\theta_k$ blocks
and  on at most $\theta_k \cdot \theta_{k-1}$ lines of $\cL_\cS$.
\end{lemma}
\begin{proof}
Consider a block $S_0$ not through $P$ (such an $S_0$ exists since
$\cS$ is not a sunflower).
Every block through $P$ contains a point $Q$ of $S_0$ and every line $PQ$ 
 with $Q \in S_0$ 
is
contained in at most one block. In this way we find at most $\theta_k$ blocks that contain $P$.
The lemma follows since the number of lines through a point in a
$k$-space is $\theta_{k-1}$.
\end{proof}

From now on we distinguish `rich' and `poor' points and lines in $\cP_\cS$ and $\cL_\cS$.
First we give the definition, then we continue with some counting arguments.

\begin{definition} \label{defrijkpunt}
Suppose $c,d$ are constants between $s$ and $1$.

A point $P \in \cP_\cS$ is \emph{$c$-rich} if it is included in more than $(1 - c) \theta_k$ blocks. A point is \emph{$c$-poor} if it is not $c$-rich. 

A line $l\in \cL_\cS$ is \emph{$(c,d)$-rich} if it contains more than $(1 - d)(q + 1)$ $c$-rich points.
\end{definition}

We will call $c$-rich and $c$-poor points, and $(c,d)$-rich lines 
\emph{rich} and \emph{poor}
points, and \emph{rich} lines respectively, if the constants $c$ and $d$ are clear
from the context.

\begin{lemma} \label{rijkpuntpersolid}
For the number $r$ of $c$-rich points in a block, we find
\[
r \ge r_0 = \left( 1 - \frac {s}{c} \right) \theta_k\;.
\]
\end{lemma}
\begin{proof}
Fix a block $S_0$ in $\cS$, and count the number of elements in $\cS$ that intersect $S_0$ in a point: through every rich point $P$ of $S_0$ there are at most $\theta_k - 1$  blocks different from $S_0$, by Lemma \ref{puntoprecht}.
Through every line spanned by $P$ and a point of such a block,
there is at most one block.
\par Every poor point of $S_0$ lies in at most $(1 - c) \theta_k - 1$  other blocks by definition.
We double-count point-block pairs $(P,S)$ with $P\in S$ where $P\in S_0$ and $S\neq S_0$ to obtain the following inequality:
\begin{align*}
	&& r ( \theta_k - 1) + (\theta_k - r) \left( (1 - c) \theta_k - 1 \right)%
	&\ge |\cS|-1 \\
	&\Leftrightarrow\quad  & r (\theta_k - 1 - (1 - c) \theta_k + 1)%
	&\ge (1 - s) \theta_k^2 - 1 - (1 - c) \theta_k^2 + \theta_k \\
	&\Rightarrow\quad  & r c \theta_k &\ge (c - s) \theta_k^2 \\
	&\Leftrightarrow\quad  & r &\ge \left(1 - \frac {s}{c}\right) \theta_k\;.\qedhere
\end{align*}
\end{proof}


\begin{lemma}\label{numberrichlines}
A block contains at least
\[
\frac {\theta_k \theta_{k-1}}{q+1} \cdot \left( 1 - \frac s{cd} \right)
\]
$(c,d)$-rich lines and the total number of $(c,d)$-rich lines is at least $(1 - s) \theta_k^2$ times this number.
\end{lemma}
\begin{proof}
Consider a block $S_0 \in \cS$ and let $\beta$ denote the number of poor lines in $S_0$. By counting pairs ($P$, $l$), with $P$ a rich point in $S_0$, $l$ a line in $S_0$ and $P\in l$, we find:
\[
\left( \qbin {k+1}2 - \beta \right) (q + 1) + \beta (1 - d)(q + 1)%
\ge r_0 \theta_{k-1} = \left( 1 - \frac sc \right) \theta_k \theta_{k-1},
\]
which gives
\[
\beta \le \frac {s \theta_k \theta_{k-1}}{c d (q + 1)}.
\]
Hence, an element of $\cS$ contains at least $\qbin{k+1}{2}-\beta\geq \frac{\theta_{k}\theta_{k-1}}{q+1}-\frac {s \theta_k \theta_{k-1}}{c d (q + 1)}$ elements.
\end{proof}

\begin{commentaar}\label{opmerkingcs}
In order to get a useful bound in the previous lemma, we need values of $s, c$ and $d$ such that 
$s \leq c d$. Later we will see that the values that we use for $c$ and $d$ satisfy these inequalities.
\end{commentaar}

We continue with a lemma that will be useful to prove the Main Lemma and the theorems in the following section.

\begin{lemma}\label{lem:richlinesthroughtwo}
	The average number of $(c,d)$-rich lines meeting two distinct blocks $S_1,S_2$ 
 in a $c$-rich point different from $S_1\cap S_2$ is at least
\[
f(s) = \theta_k \theta_{k-1} q \frac {1 - d}{1 - s}%
\left( 1 - \frac s{cd} \right) \left( 1 - c - \frac 1{\theta_k} \right)^2 %
\left( 1 - d - \frac dq \right).
\]
\end{lemma}
\begin{proof}
We count triples $(S_1,S_2,r)$ where $r$ is a rich line connecting
a rich point in $S_1\setminus S_2$ with a rich point in $S_2\setminus S_1$. Let $\rho_{\{S_1,S_2\}}$, $S_1,S_2\in \cS$, $S_1\neq S_2$, be the number of rich lines meeting both $S_1\setminus S_2$ and $S_2\setminus S_1$ in a rich point. We define $\rho(s)$ as the average of the values $\rho_{\{S_1,S_2\}}$ with $S_1,S_2\in \cS$ and $S_1\neq S_2$. On the one hand the number of triples equals
\[
(1 - s) \theta_k^2 \left( (1 - s) \theta_k^2 - 1 \right) \rho(s) \leq  (1 - s)^2 \theta_k^4 \cdot \rho(s)\;.
\]
On the other hand, the number of triples is at least 
\[
(1 - s) \theta_k^2  \frac {\theta_k\theta_{k-1}}{q+1}  \left( %
1 - \frac s{cd} \right) \cdot(1 - d)(q + 1)\cdot((1 - d)q - d)\cdot((1 - c) \theta_k - 1)^2\;,
\]
as by Lemma \ref{numberrichlines}, there are at least
$\ds (1 - s) \theta_k^2 \cdot \frac {\theta_k\theta_{k-1}}{q+1} \cdot \left( 1 - \frac s{cd} \right)$ rich lines, and on a rich line there are at least $(1 - d)(q + 1)((1 - d)q - d)$ possibilities for an ordered pair of two distinct rich points $P_1,P_2$. Through those points we find at least $((1 - c) \theta_k - 1)^2$ possibilities for the blocks $S_1$ and $S_2$ (not
containing the line $P_1P_2$). This gives that the average $\rho(s)$ is at least $f(s)$.
\end{proof}

\section{Main Lemma and results}\label{section3}

Using the combinatorial lemmas in the previous section, the main goal in this section is to find a an upper bound on $(1 - s)$, as a function of the field size $q$. We start with the Main Lemma, that will be the basis of the theorems at the end of this section.

\begin{mlemma}\label{mlemma}
Let $\cS$ be a $(k+1,1)$-SCID in $\PG(n,q)$ with $|\cS|=(1-s)\theta_k^2$, $k\geq 3$, that is not a sunflower. For all values $0 < s < c, d < 1$, we have the following inequality:
\begin{multline}\label{belangrijkevgl}
\left( 1 - \frac s{cd} \right) (1 - d) (1 - c) \left( 1 - c -%
\frac 1{q^3} \right)^2 \left( 1 - d - \frac dq \right) %
\left( 1 - d - \frac {1+d}q \right) q  \\
\le (1 - s)^2 + \frac {1 - s}{q}.
\end{multline}
\end{mlemma}
\begin{proof}
Consider a pair of different blocks $S_1,S_2$ such that there are at least $f(s)$ distinct $(c,d)$-rich lines connecting a point of $S_1\setminus S_2$ and a point of $S_2\setminus S_1$. Note that such a pair is guaranteed to exist by Lemma \ref{lem:richlinesthroughtwo}. Then, the $2k$-space $T=\langle S_1,S_2\rangle$ contains at least 
\[
f(s) \cdot \frac{(1 - d)(q + 1) - 2}{q} = \left(1 - d - \frac{1 + d}{q}\right) f(s)
\]
rich points, since every point $P$ in the $2k$-space $T$ not in the union $S_1\cup S_2$ lies on at most $q$ such connecting lines. Indeed, each such line is contained in the plane $\langle P, S_1\rangle\cap\langle P, S_2\rangle$, and in this plane there are $q$ lines through $P$ that do not contain $\cS_1 \cap \cS_2$. Each of these $q$ lines might or might not be rich. 

Since the dual of a $(k+1,1)$-SCID in a $2k$-space is a partial $(k-1)$-spread in this $2k$-space, we have that a $2k$-space contains at most $\lfloor \theta_{2k}/\theta_{k-1} \rfloor = q^{k+1} + q$ blocks. On the other hand, $T$ contains at most $\theta_{k-1}$ points from each block not contained in $T$. Hence the number of pairs $(P,S)$, with $P \in \langle S_1, S_2 \rangle$ a rich point in the block $S$ is at least $(1 - d - (1+d)/q) f(s) (1 - c)\theta_k$ and at most $(q^{k+1} + q) \theta_k + \left( (1 - s) \theta_k^2 - (q^{k+1} + q) \right) %
\theta_{k-1}$.
Hence
\begin{align*}
&\left(1 - d - \frac{1+d}{q}\right) (1 - c) f(s)\theta_k\\
&\qquad\qquad\leq 
(q^{k+1} + q) \theta_k + \left( (1 - s) \theta_k^2 - (q^{k+1} + q) \right)\theta_{k-1}\\
\Rightarrow\quad &\left(1 - d - \frac{1+d}{q}\right) (1 - c) \frac{f(s)}{\theta_k \theta_{k-1}} 
\leq 1-s+
\frac{q^{k+1} + q}{\theta_k^2 \theta_{k-1}}q^k \leq 1 - s + \frac 1{q^{k-2}}\;.
\end{align*}
The last inequality follows since $q^k (q^{k+1} + q) \le q^{2-k} \theta_k^2 \theta_{k-1}^{\phantom 2}$.
This implies that
\begin{multline}
\left( 1 - \frac s{cd} \right) (1 - d) (1 - c) \left( 1 - c -%
\frac 1{\theta_k} \right)^2 \left( 1 - d - \frac dq \right) %
\left( 1 - d - \frac {1+d}q \right) q  \\
\le (1 - s)^2 + \frac {1 - s}{q^{k-2}}\;,\label{blaqk}
\end{multline}
which proves the lemma since $k\geq 3$.\qedhere
\end{proof}

\begin{gevolg}\label{gevolg}
Let $\cS$ be a $(k+1,1)$-SCID in $\PG(n,q)$ with $|\cS|=(1-s)\theta_k^2$, $k\geq 3$, that is not a sunflower. 
Suppose that 
\begin{align*}
	\left(\frac{1}{q}-\frac{B(q,c,d)}{cd}\right)^2-4B(q,c,d)\left(\frac{1}{cd}-1\right)\geq 0.
\end{align*}
Then we have, for all values $0 < s < c, d < 1$, that
{\footnotesize\begin{align*}
	&(1-s)\leq F(q,c,d)=\frac{1}{2}\left( \frac{B(q,c,d)}{cd}-\frac{1}{q}-\sqrt{\left(\frac{1}{q}-\frac{B(q,c,d)}{cd}\right)^2-4B(q,c,d)\left(\frac{1}{cd}-1\right)}\right)\\ \text{ or } &(1-s)\geq G(q,c,d)=\frac{1}{2}\left( \frac{B(q,c,d)}{cd}-\frac{1}{q}+\sqrt{\left(\frac{1}{q}-\frac{B(q,c,d)}{cd}\right)^2-4B(q,c,d)\left(\frac{1}{cd}-1\right)}\right)
	\end{align*}}
with $B(q,c,d)=(1 - d) (1 - c) \left( 1 - c -%
\frac 1{q^3} \right)^2 \left( 1 - d - \frac dq \right) %
\left( 1 - d - \frac {1+d}q \right) q. $
\end{gevolg}
\begin{proof}
Using inequality (\ref{belangrijkevgl}) from the Main Lemma, we immediately find the following quadratic inequality \[(1-s)^2+\left(\frac{1}{q}-\frac{B(q,c,d)}{cd}\right)\cdot (1-s)+B(q,c,d)\left(\frac{1}{cd}-1\right)\geq 0,\]  
which proves the corollary.
\end{proof}

From now on, we put $c(q)=d(q)=1-\frac{1}{\sqrt[6]{q}}-\frac{1}{2\sqrt[3]{q}}$. Since $c$ and $d$ must be non negative by definition, we have to assume that $q\geq 7$. We denote $c(q), F(q,c(q),c(q)), G(q, c(q),c(q))$ and $B(q,c(q),c(q))$ by $c_q, F_q, G_q$ and $B_q$ respectively. We first give two lower bounds on $B_q$.

\begin{lemma}\label{lemmaonB}
	Let $t=\sqrt[6]{q}$, $q\geq 7$,  then 
	\begin{align}
		B_q &> \left(1+\frac{1}{2t}\right)^2 \left(1+\frac{1}{2t}-\frac{1}{t^4}\right)^2 \left(1+\frac{1}{2t}-\frac{1}{t^5}\right)\left(1+\frac{1}{2t}-\frac{2}{t^5}\right) \label{lemmaB1} \\
		B_q & > \left(1+\frac{1}{2t}\right)^2 \left( 1+\frac{1}{3t}  \right)^2. \label{lemmaB2}
	\end{align}
\end{lemma}
\begin{proof}
	By using the equality $c_q=c_{t^6}=1-\frac{1}{t}-\frac{1}{2t^2}$ and $t\geq \sqrt[6]{7}$, we have
	\begin{align*}
B_q&=(1 - c_q)^2 \left( 1 - c_q -\frac 1{q^3} \right)^2 \left( 1 - c_q - \frac{c_q}{q} \right) \left( 1 - c_q - \frac {1+c_q}q \right) q\\
&=\left(\frac{1}{t}+\frac{1}{2t^2}\right)^2 \left(\frac{1}{t}+\frac{1}{2t^2} -\frac{1}{t^{18}} \right)^2 \left(\frac{1}{t}+\frac{1}{2t^2}-\frac{1}{t^6}+\frac{1}{t^7}+\frac{1}{2t^8} \right)\\
&\qquad\qquad \left( \frac{1}{t}+\frac{1}{2t^2}- \frac{2}{t^6}+\frac{1}{t^7}+\frac{1}{2t^8} \right) t^6\\
&=\left(1+\frac{1}{2t}\right)^2 \left( 1+\frac{1}{2t} -\frac{1}{t^{17}} \right)^2 \left( 1+\frac{1}{2t}- \frac{1}{t^5}+\frac{1}{t^6}+\frac{1}{2t^7} \right)\\ 
&\qquad\qquad\left( 1+\frac{1}{2t}- \frac{2}{t^5}+\frac{1}{t^6}+\frac{1}{2t^7} \right)\;.
\end{align*}
Using this expression for $B_q$, we can check that the following two inequalities are true for all $t\geq \sqrt[6]{7}$, and so, for all $q\geq 7$. 
\begin{align*}
B_q &> \left(1+\frac{1}{2t}\right)^2 \left(1+\frac{1}{2t}-\frac{1}{t^4}\right)^2 \left(1+\frac{1}{2t}-\frac{1}{t^5}\right)\left(1+\frac{1}{2t}-\frac{2}{t^5}\right), \ \text{and} \\
B_q &> \left(1+\frac{1}{2t}\right)^2 \left( 1+\frac{1}{3t}  \right)^2\;. \qedhere
\end{align*}
\end{proof}

We continue by investigating the condition $\left(\frac{1}{q}-\frac{B_q}{c_q^2}\right)^2-4B_q\left(\frac{1}{c_q^2}-1\right)\geq 0$ from Corollary \ref{gevolg}, proving for which values of $q\geq 7$ it is valid. Or equivalently, for which values of $q$, the argument of the square root in $F_q$ and $G_q$ is non-negative. 

\begin{lemma}\label{onderwortel}
	For $q\geq 7$ it is true that $\left(\frac{1}{q}-\frac{B_q}{c_q^2}\right)^2-4B_q\left(\frac{1}{c_q^2}-1\right)\geq 0$, with $B_q= (1 - c_q)^2 \left( 1 - c_q -\frac 1{q^3} \right)^2 \left( 1 - c_q - \frac{c_q}{q} \right)\left( 1 - c_q - \frac {1+c_q}q \right) q$ and $c_q=1-\frac{1}{\sqrt[6]{q}}-\frac{1}{2\sqrt[3]{q}}$.
\end{lemma}
\begin{proof}
	Note that it follows from Lemma \ref{lemmaonB} that $B_q>0$ if $q\geq 7$ (we will use this on the third line). Suppose that the inequality in the statement of the lemma does not hold.
	Then we have
	\begin{align*}
	&&& \hspace{-1cm}\frac{B_q^2}{c_q^4}-\frac{2B_q}{q c_q^2}+\frac{1}{q^2}< 4B_q\left(\frac{1}{c_q^2}-1\right) \\
	&\Rightarrow&& \hspace{-1cm} \frac{B_q^2}{c_q^4}< 2B_q\left(\frac{2}{c_q^2}-2+\frac{1}{q c_q^2}\right) \\
	&\xLeftrightarrow{B_q> 0}&&\hspace{-1cm}  B_q< 2c_q^2\left(2(1-c_q^2)+\frac{1}{q}\right) \\
	&\xLeftrightarrow{t=\sqrt[6]{q}}&&\hspace{-1cm} B_{t^6} <2\left(1-\frac{1}{t}-\frac{1}{2t^2}\right)^2\left( \frac{4}{t}-\frac{2}{t^3}-\frac{1}{2t^4} +\frac{1}{t^6}\right)\\	
	& \xRightarrow{(\ref{lemmaB2})}&&\hspace{-1cm}  \left(1+\frac{1}{2t}\right)^2 \left( 1+\frac{1}{3t}  \right)^2  <2\left(1-\frac{1}{t}\right)^2\left( \frac{4}{t}-\frac{1}{t^3}\right) \\
	& \Leftrightarrow&&\hspace{-1cm}  \left(t+\frac 12\right)^2\left(t+\frac 13\right)^2<2\left(t-1\right)^2\left( 4t-\frac{1}{t}\right) \\
	& \Leftrightarrow &&\hspace{-1cm} t^4+\frac{5}{3}t^3+\frac{37}{36}t^2+\frac{5}{18}t+\frac{1}{36}<8t^3-16t^2+6t+4-\frac{2}{t}\\
	& \Leftrightarrow &&\hspace{-1cm} t^4-\frac{19}{3}t^3+\frac{613}{36}t^2-\frac{103}{18}t-\frac{143}{36}+\frac{2}{t}<0. 
	\end{align*}
	The last inequality gives a contradiction for all values of $t\geq \sqrt[6]{7}$, and so for all $q\geq 7$,  which proves the lemma.
\end{proof}

Now we prove that $G_q> 1$. This implies that the first bound in Corollary \ref{gevolg} holds,
since $0<s<1$.
	
\begin{lemma}\label{boundonG}
	For $q\geq 7$, it is true that 
	\begin{align*}
		G_q=\frac{1}{2}\left(\frac{B_q}{c_q^2}-\frac{1}{q}+\sqrt{\left(\frac{1}{q}-\frac{B_q}{c_q^2}\right)^2-4B_q\left(\frac{1}{c_q^2}-1\right)}\right)> 1
	\end{align*}
	with
	\begin{align*}
		B_q&= (1 - c_q)^2 \left( 1 - c_q -%
		\frac 1{q^3} \right)^2 \left( 1 - c_q - \frac{c_q}{q} \right) \left( 1 - c_q - \frac {1+c_q}q \right) q\ \text{ and} \\ c_q&=1-\frac{1}{\sqrt[6]{q}}-\frac{1}{2\sqrt[3]{q}}\;.
	\end{align*}
\end{lemma}
\begin{proof}
	We have to prove that 
	\begin{align*}
	\frac{B_q}{c_q^2}-\frac{1}{q}+\sqrt{\left(\frac{1}{q}-\frac{B_q}{c_q^2}\right)^2-4B_q\left(\frac{1}{c_q^2}-1\right)}> 2\;.
	\end{align*}
	For all values of $q\geq 7$ such that $2-\frac{B_q}{c_q^2}+\frac{1}{q}<0$, the previous inequality is true. If $2-\frac{B_q}{c_q^2}+\frac{1}{q}\geq 0$, then it is equivalent to proving that
	\begin{align*}
		&\quad \left(\frac{1}{q}-\frac{B_q}{c_q^2}\right)^2-4B_q\left(\frac{1}{c_q^2}-1\right) > 4+4\left(\frac{1}{q}-\frac{B_q}{c_q^2}\right)+\left(\frac{1}{q}-\frac{B_q}{c_q^2}\right)^2\\
		\Leftrightarrow& \quad -\frac{B_q}{c_q^2}+B_q > 1+\frac{1}{q}-\frac{B_q}{c_q^2}\\
		\Leftrightarrow & \quad B_q > 1+\frac{1}{q}\;.
	\end{align*}
	Set $t=\sqrt[6]{q}$. From Lemma \ref{lemmaonB}(\ref{lemmaB2}), we know that it is sufficient to prove the following inequality: 
	\begin{align*}
		&\quad \left(1+\frac{1}{2t}\right)^2 \left( 1+\frac{1}{3t} \right)^2  > 1+\frac{1}{t^6}\\
		    \Leftrightarrow &\quad  t^4+\frac{5}{3}t^3+\frac{37}{36}t^2+\frac{5}{18}t+\frac{1}{36}> t^4+\frac{1}{t^2}\\
		    \Leftrightarrow & \quad \frac{5}{3}t^3+\frac{37}{36}t^2+\frac{5}{18}t+\frac{1}{36}-\frac{1}{t^2}>0\;.
	\end{align*}
	This last inequality is true for $t\geq \sqrt[6]{7}$, and so for $q\geq 7$, which proves the lemma.
\end{proof}
	
\begin{theorem}\label{maintheorem}
	A $(k+1,1)$-SCID in $\PG(n,q)$, $k\geq 3, q\geq 7$,  that has more than $F_q\theta_k^2$ elements, is a sunflower, with
	\begin{align*}
	 F_q \nonumber  = \frac{1}{2} \left(
	\frac{B_q}{c_q^2}-\frac{1}{q}-\sqrt{\left(\frac{1}{q}-\frac{B_q}{c_q^2}\right)^2-4B_q\left(\frac{1}{c_q^2}-1\right)} \right)
	\end{align*}
	and 
\begin{align*}
	B_q&= (1 - c_q)^2 \left( 1 - c_q -%
	\frac 1{q^3} \right)^2 \left( 1 - c_q - \frac{c_q}{q} \right) \left( 1 - c_q - \frac {1+c_q}q \right) q\;, \\ c_q&=1-\frac{1}{\sqrt[6]{q}}-\frac{1}{2\sqrt[3]{q}}\;.
	\end{align*}
	In particular, we have that a $(k+1,1)$-SCID in $\PG(n,q)$, with more than $\left(\frac{2}{\sqrt[6]{q}}+\frac{4}{\sqrt[3]{q}}-\frac{5}{\sqrt{q}}\right)\theta_k^2$ elements is a sunflower. 
\end{theorem}
\begin{proof} 
	From Corollary \ref{gevolg}, Lemma \ref{onderwortel} and Lemma \ref{boundonG},  we know that $F_q \theta_k^2$ gives an upper bound on the size $|\cS|=(1-s)\theta_k^2$ of a $(k+1,1)$-SCID, with $\cS$ not a sunflower.  Hence, a $(k+1,1)$-SCID with more than $F_q \theta_k^2$ elements is a sunflower.	
	We have to prove that 
	\begin{align*} & \quad F_q\leq \frac{2}{\sqrt[6]{q}}+\frac{4}{\sqrt[3]{q}}-\frac{5}{\sqrt{q}}\\
	\Leftrightarrow &  \quad \frac{B_q}{c_q^2}-\frac{1}{q}-\sqrt{\left(\frac{1}{q}-\frac{B_q}{c_q^2}\right)^2-4B_q\left(\frac{1}{c_q^2}-1\right)} \leq \frac{4}{\sqrt[6]{q}}+\frac{8}{\sqrt[3]{q}}-\frac{10}{\sqrt{q}}
	\end{align*}
	If $\frac{B_q}{c_q^2}-\frac{1}{q}-\frac{4}{\sqrt[6]{q}}-\frac{8}{\sqrt[3]{q}}+\frac{10}{\sqrt{q}}\leq0$, then this is true for all values of $q\geq 7$. If $\frac{B_q}{c_q^2}-\frac{1}{q}-\frac{4}{\sqrt[6]{q}}-\frac{8}{\sqrt[3]{q}}+\frac{10}{\sqrt{q}}> 0$, then it is equivalent to proving that 
	\begin{align*}
	&&& \left(\frac{1}{q}-\frac{B_q}{c_q^2}\right)^2-4B_q\left(\frac{1}{c_q^2}-1\right) \\ &&&\qquad \geq \left(\frac{4}{\sqrt[6]{q}}+\frac{8}{\sqrt[3]{q}}-\frac{10}{\sqrt{q}}\right)^2+2\left(\frac{4}{\sqrt[6]{q}}+\frac{8}{\sqrt[3]{q}}-\frac{10}{\sqrt{q}}\right)\left(\frac{1}{q}-\frac{B_q}{c_q^2}\right) \\ &&&\qquad\qquad +\left(\frac{1}{q}-\frac{B_q}{c_q^2}\right)^2 \\
	&\Leftrightarrow &&  B_q\left(-\frac{1}{c_q^2}+1+\frac{1}{c_q^2}\left(\frac{2}{\sqrt[6]{q}}+\frac{4}{\sqrt[3]{q}}-\frac{5}{\sqrt{q}}\right)\right) \\ &&&\qquad \geq \left(\frac{2}{\sqrt[6]{q}}+\frac{4}{\sqrt[3]{q}}-\frac{5}{\sqrt{q}}\right)^2+\frac 1q\left(\frac{2}{\sqrt[6]{q}}+\frac{4}{\sqrt[3]{q}}-\frac{5}{\sqrt{q}}\right)\\
	&\Leftrightarrow &&  B_q\left(c_q^2-1+\frac{2}{\sqrt[6]{q}}+\frac{4}{\sqrt[3]{q}}-\frac{5}{\sqrt{q}}\right)\\ &&&\qquad \geq c_q^2 \left( \left(\frac{2}{\sqrt[6]{q}}+\frac{4}{\sqrt[3]{q}}-\frac{5}{\sqrt{q}}\right)^2+\frac 1q\left(\frac{2}{\sqrt[6]{q}}+\frac{4}{\sqrt[3]{q}}-\frac{5}{\sqrt{q}}\right) \right)\\
	&\xLeftrightarrow{t=\sqrt[6]{q}} &&  B_{t^6}\left(\frac{1}{4t^4}+\frac{4}{t^2}-\frac{4}{t^3}\right)\\
	&&&\qquad\geq \left(1-\frac{1}{t}-\frac{1}{2t^2}\right)^2 \left( \left(\frac{2}{t}+\frac{4}{t^2}-\frac{5}{t^3}\right)^2+\frac{1}{t^6}\left(\frac{2}{t}+\frac{4}{t^2}-\frac{5}{t^3}\right) \right)\;.
	\end{align*}
	Due to Lemma \ref{lemmaonB}$(\ref{lemmaB1})$, it is sufficient to prove that
	\begin{multline*}
 		\left(1+\frac{1}{2t}\right)^2 \left(1+\frac{1}{2t}-\frac{1}{t^4}\right)^2 \left(1+\frac{1}{2t}-\frac{1}{t^5}\right)\left(1+\frac{1}{2t}-\frac{2}{t^5}\right)\left(\frac{1}{4t^4}+\frac{4}{t^2}-\frac{4}{t^3}\right) \\\geq \left(1-\frac{1}{t}-\frac{1}{2t^2}\right)^2 \left( \left(\frac{2}{t}+\frac{4}{t^2}-\frac{5}{t^3}\right)^2+\frac{1}{t^6}\left(\frac{2}{t}+\frac{4}{t^2}-\frac{5}{t^3}\right) \right),	
 	\end{multline*}
 	equivalently that
 	\begin{multline*}
 	\frac{157}{4t^4} +\frac{95}{4t^5} 
 	-\frac{2165}{16t^6} +\frac{173}{8 t^7}+\frac{1411}{64t^8}+\frac{383}{64t^9}+\frac{1313}{256t^{10}} \\+ \frac{69}{2t^{11}} + \frac{1177}{32 t^{12}} -\frac{37}{8t^{13}} -\frac{3315}{128 t^{14}}  -\frac{219}{8t^{15}}-\frac{1631}{64t^{16}}+\frac{3}{32 t^{17}}\\+\frac{557}{32 t^{18}}+ \frac{151}{16 t^{19}} +\frac{293}{32 t^{20}} -\frac{1}{8t^{21}}-\frac{11}{2t^{22}}-\frac{3}{2 t^{23}}+\frac{1}{8 t^{24}}\geq 0\;.
 	\end{multline*}
 	This inequality is true for all $t\geq \sqrt[6]{7}$, and so for  $q\geq 7$. So, a $(k+1,1)$-SCID in $\PG(n,q)$, with at least $\left(\frac{2}{\sqrt[6]{q}}+\frac{4}{\sqrt[3]{q}}-\frac{5}{\sqrt{q}}\right)\theta_k^2$ elements, has more than $F_q\theta_k^2$ elements. This implies that this SCID is a sunflower, which proves the theorem. 
\end{proof}

Note that the bound $1-s \leq \frac{2}{\sqrt[6]{q}}+\frac{4}{\sqrt[3]{q}}-\frac{5}{\sqrt{q}}$ only gives an improvement for the sunflower bound for $q\geq 473$, and so, it is useful for large values of $q$. For $q\geq 473$, this bound is also an improvement on the bound in \cite{sunflowern-2}. For fixed, smaller values of $q$ an improved sunflower bound can be found by investigating the bound $1-s \leq F_q$. This bound gives an improvement on the sunflower bound if $F_q<1-\frac{1}{\theta_k}+\frac{1}{\theta_k^2}$.  For $k=3$ and $k=4$, this is the case for $q\geq 9$ and $q\geq 8$ respectively. For $k>4$ we have that $F_q<1-\frac{1}{\theta_k}+\frac{1}{\theta_k^2}$, if $F_q<1-\frac{1}{\theta_5}$, which is the case for $q\geq 7$. For these values of $q$ and $k$, the bound $1-s\leq F_q$ improves the bound in \cite{sunflowern-2}.

 \begin{table}[h!] \centering
 	\begin{tabular}{|c|c|c|}
 		\hline 
 		$q$ & $F_q$& $\frac{2}{\sqrt[6]{q}}+\frac{4}{\sqrt[3]{q}}-\frac{5}{\sqrt{q}}$\\  \hline
 		$2^4$ & 0.97698136 & 1.59732210 \\
 		$2^6$ & 0.89046942 & 1.37500000 \\
 		$2^8$ & 0.78319928 & 1.11116105\\
 		$2^{10}$ & 0.67282525 &0.87056078  \\
 		$2^{12}$ & 0.56493296 &0.67187500\\
 		$2^{14}$ & 0.46301281 & 0.51527789\\
 		$2^{16}$ & 0.37118406 & 0.39466158\\
 		$2^{18}$ & 0.29280283 & 0.30273438 \\
 		$2^{20}$ & 0.22886576 & 0.23291485 \\
 		\hline 
 	\end{tabular}\caption{Upper bound $F_q$ and $\frac{2}{\sqrt[6]{q}}+\frac{4}{\sqrt[3]{q}}-\frac{5}{\sqrt{q}}$ on $(1-s)$ for specific values of $q$.}\label{tabel}
 \end{table}
 
In Table \ref{tabel} we give the values of the upper bound $F_q$ and $\frac{2}{\sqrt[6]{q}}+\frac{4}{\sqrt[3]{q}}-\frac{5}{\sqrt{q}}$ on $(1-s)$,  for some specific values $q$. The values in this table confirm that the bound $\frac{2}{\sqrt[6]{q}}+\frac{4}{\sqrt[3]{q}}-\frac{5}{\sqrt{q}}$ is a good approximation for $F_q$ for large values of $q$. 

Note that for fixed values of $k$ and $q$ there is a possibility to find a slightly better bound than the bound $F_q$, by using our techniques. Given the fixed values for $k$ and $q$ in inequality $(\ref{blaqk})$, we can choose the values of $c$ and $d$ such that we get the optimal bound for $1-s$. We describe this technique in the example below.
\begin{example}
	Suppose that $q=2^8=256$ and $k=5$, then we find from $(\ref{blaqk})$, that 
	\begin{align*}
		&\left( 1 - \frac s{cd} \right) (1 - d) (1 - c) \left( 1 - c -%
		\frac 1{\theta_5} \right)^2 \left( 1 - d - \frac{d}{2^8} \right) %
		\left( 1 - d - \frac {1+d}{2^8} \right) 2^8  \\ &\hspace{8.4cm}
		\leq (1 - s)^2 + \frac {1 - s}{2^{24}}, 
		\\ &\Leftrightarrow\left( 1 - \frac s{cd} \right) B(c,d)
		\leq (1 - s)^2 + \frac {1 - s}{2^{24}}\\ &\Leftrightarrow (1-s)^2+(1-s) \left( \frac{1}{2^{24}} - \frac{B(c,d)}{cd} \right) -\left(1- \frac{1}{cd}\right)B(c,d)
		\geq 0 \\ &\Leftrightarrow (1-s)\leq\frac{1}{2}\left( \frac{B(c,d)}{cd}-\frac{1}{2^{24}}-\sqrt{\left(\frac{1}{2^{24}}-\frac{B(c,d)}{cd}\right)^2-4\left(\frac{1}{cd}-1\right)B(c,d)}\right)
	\end{align*}
	with $B(c,d)=(1 - d) (1 - c) \left( 1 - c -%
	\frac {1}{\theta_5(2^8)} \right)^2 \left( 1 - d - \frac{d}{2^8} \right) %
	\left( 1 - d - \frac {1+d}{2^8} \right) 2^8. $
	By using a computer algebra package, we find a very good bound on $1-s$ for $c=0.53152285$ and $d=0.5294$. For these values we find the bound $1-s\leq 0.7825095$. Hence this gives a small improvement on the bound $1-s\leq F_q=0.78319928$, for which we used $c(2^8)=d(2^8)=0.5244047$.
	Note that the bound, given by the Sunflower Theorem \ref{sunny}, and the bound given in \cite{sunflowern-2} are both larger than $0.99999999 \theta_k^2$ for $q=2^8=256$ and $k=5$. This indicates that our new bound is a clear improvement. 
\end{example}

\subsection*{Acknowledgements}
The research of Jozefien D'haeseleer is supported by the FWO (Research Foundation Flanders). 
We would like to thank our colleague Lins Denaux for proof-reading this article in detail.

\end{document}